\documentclass[reqno,12pt]{amsart}
\usepackage[colorlinks=true,urlcolor=blue,pdfstartview=FitH, linkcolor=blue,citecolor=blue,bookmarksopen, bookmarksnumbered={true}]{hyperref}
\usepackage[top=2.5cm,bottom=2.5cm,left=2.5cm,right=2.5cm]{geometry}

\numberwithin{equation}{section}
\newtheorem{theorem}{Theorem}

\newtheorem{definition}[theorem]{Definition}
\newtheorem{example}[theorem]{Example}

\newtheorem{proposition}[theorem]{Proposition}
\newtheorem{remark}[theorem]{Remark}

\begin{document}
\title[Generalization of the Banach contraction principle]{Generalization of the Banach contraction principle}

\author[M. Berzig]{Maher Berzig}

\address{Maher Berzig,  \newline
\indent Tunis College of Sciences and Techniques,\newline
\indent 5 Avenue Taha Hussein, Tunis University,  Tunisia}
\email{maher.berzig@gmail.com}




\begin{abstract}
We introduce the concept of shifting distance functions, and we establish a new fixed point theorem which generalizes the Banach contraction principle. An  example is provided to illustrate  our result. 
 \end{abstract}

\maketitle

\section{Introduction and Preliminaries}

The Banach contraction principle \cite{B} provides the most  simple and efficient tools in fixed point theory. This principle is used to studying the existence and uniqueness of solution for a wide class of  linear and nonlinear functional equations arising in pure and applied mathematics.  
\begin{theorem}[see  \cite{B}]
Let $(X, d)$ be a complete metric space and let $T : X\to X$ be a self-mapping. Suppose that there exists a constant $k\in[0,1)$  such that
$$
d(Tx,Ty) \le kd(x,y)\quad \text{ for all } x,y \in X.
$$
Then $T$ has a unique fixed point, that is, there exists $x^*\in X$ such that $Tx^*=x^*$.
\end{theorem}
The above theorem has been generalized and extended in different directions see for example  \cite{AGB,B,Be,BK,BS,BW,C,DC,HR,K,KSS,MK,N,R,RR}, and references therein. Khan et al. introduced  the concept of  altering distance functions as follows. 
\begin{definition}[see \cite{KSS}]
A functions $\psi : [0,\infty)\to [0,\infty)$ is called an altering distance function if the following properties are satisfied:
\begin{enumerate}
  \item $\psi(t)=0$ if and only if $t=0$.
  \item $\psi$ is  continuous and nondecreasing.
  \label{D1}
\end{enumerate}
\end{definition}
We state the result of Khan et al.  in the following.
\begin{theorem}[see \cite{KSS}]
Let $(X, d)$ be a complete metric space and let $T : X\to X$ be a self-mapping. Suppose that there exist an altering distance functions $\psi$ and a constant $c\in[0,1)$  such that
$$
\psi(d(Tx,Ty)) \le c\psi(d(x,y))\quad \text{ for all } x,y \in X.
$$
Then $T$ has a unique fixed point.
\label{Kh}
\end{theorem}

On the other hand, Alber and Guerre-Delabriere \cite{AGB} have introduced the concept of weakly contractive maps and established a generalization of the contraction principle in Hilbert spaces.  Next, Rhodes extended and generalized this concept  to be valid even in complete metric spaces as follows.  

\begin{theorem} [see \cite{R}]
Let $(X, d)$ be a complete metric space and let $T : X\to X$ be a self-mapping. Suppose that there exist an altering distance functions $\psi$  such that
$$
\psi(d(Tx,Ty)) \le d(x,y)- \psi(d(x,y))\quad \text{ for all } x,y \in X.
$$
Then $T$ has a unique fixed point.
\label{Ro}
\end{theorem}

However, Dutta and Choudhury  proved  a generalized version of Theorem \ref{Ro} as follows.
\begin{theorem}[see \cite{DC}]
Let $(X, d)$ be a complete metric space and let $T : X\to X$ be a self-mapping. Suppose that there exist two altering distance functions $\psi$ and $\varphi$ such that
\begin{equation}\label{cnt}
\psi(d(Tx,Ty)) \le \psi(d(x,y))-\varphi(d(x,y))\quad \text{ for all } x,y \in X.
\end{equation}
Then $T$ has a unique fixed point.
\label{DC}
\end{theorem}

In this paper, we obtain a new generalization of the Banach contraction principle. Our theorem generalizes also the results of Dutta and Choudhury \cite{DC}. An example is given to illustrate our result.

\section{Main result}

Let us  introduce some definition before we state our main result.
\begin{definition}
Let $\psi,\phi:[0,+\infty)\to {\mathbb R}$ be two functions. The pair of functions  $(\psi,\phi)$ is said to be a pair of shifting distance functions  if the following conditions hold:
\begin{itemize}
\item[\rm(i)]  for $u, v\in [0,+\infty)$ if $\psi(u)\le\phi(v) $, then $u\le v ;$
\item[\rm(ii)]  for $\{u_n\}, \{v_n\}\subset [0,+\infty)$ with $\lim\limits_{n\to\infty}u_n=\lim\limits_{n\to\infty}v_n=w$,   if $\psi(u_n)\le\phi(v_n)$  $\forall n\ge 0$, then $w=0.$
\label{shift}
\end{itemize}
\end{definition}

\begin{example}
The conditions (i) and (ii) of the above definition are fulfilled for the functions $\psi,\phi:[0,+\infty)\to {\mathbb R}$ defined by $\psi(t)=\ln ((1+2t)/2)$ and $\phi(t)=\ln ((1+t)/2)$, $\forall t\in [0,+\infty)$.
\end{example}

Now, we are ready to state the main results.

\begin{theorem}
Let $(X,d)$ be a complete metric space and $T:X\rightarrow X$ be a mappings. Suppose that there exist a pair of shifting distance functions  $(\psi,\phi)$ such that
\begin{equation}
\psi(d(Tx,Ty))\leq \phi(d(x,y))\quad\text{for all}\quad x,y\in X.
\label{ctr}
\end{equation}
Then $T$ has a unique fixed  point in $X$.
\label{TH1}
\end{theorem}
\begin{proof}
Let $x_0\in X$, we define the sequence $\{x_n\}$ in $X$ by
$$
x_{n+1}=Tx_n,\quad n\ge0.
$$
The proof is divided into 4 steps. \\
{\it Step 1.} we shall prove that $\lim_{n \to \infty}d(x_{n+1},x_{n})= 0$.\\
We apply inequality (\ref{ctr}) for $x=x_{n-1}$ and $y=x_n$, we obtain
\begin{equation*}
\psi(u_{n})\le \phi(u_{n-1})\quad\text{for all}\quad n\ge 1.
\end{equation*}
where $u_n=d(x_{n+1},x_{n})$, which implies, by (i) from Definition (\ref{shift}), that   $\{u_n\}$ is a  decreasing sequence. Therefore, there exists $ r\ge0$ such that $\lim\limits_{n \to \infty}u_n= r$.
Using {(ii)} from Definition \ref{shift}, we obtain that $r=0$. Then,
\begin{equation}\label{lim}
\lim_{n \to \infty}d(x_{n+1},x_{n})= 0.
\end{equation}
{\it Step 2.} we shall prove that  $\{x_n\}$ is a Cauchy sequence.\\
Suppose that $\{x_n\}$ is not a Cauchy sequence. Then there exists $\varepsilon > 0$ for which we can find subsequences $\{x_{m(k)}\}$ and $\{x_{n(k)}\}$ of $\{x_n\}$ with $n(k) > m(k) > k$ such that
\begin{equation*}
  d(x_{m(k)}, x_{n(k)}) \ge \varepsilon\quad\text{and}\quad  d(x_{m(k)}, x_{n(k)-1}) < \varepsilon.
\end{equation*}
Then, we have
\begin{equation*}
\varepsilon\le  d(x_{m(k)}, x_{n(k)})\le d(x_{m(k)}, x_{n(k)-1})+d(x_{n(k)-1}, x_{n(k)}) < \varepsilon+d(x_{n(k)-1}, x_{n(k)}) .
\end{equation*}
Let $k\to \infty$ in the above inequality and using (\ref{lim}), we get
\begin{equation}\label{E4}
\lim_{k\to\infty}  d(x_{m(k)}, x_{n(k)}) =\varepsilon.
\end{equation}
Next, we have
\begin{align*}
&|d(x_{m(k)}, x_{n(k)}) - d(x_{m(k)-1}, x_{n(k)}) | \le d(x_{m(k)}, x_{m(k)-1}),\\
&|d(x_{m(k)-1}, x_{n(k)}) - d(x_{m(k)-1}, x_{n(k)-1}) | \le d(x_{n(k)}, x_{n(k)-1}).
\end{align*}
Let  $k\to\infty$ in the above inequalities and using (\ref{lim}) and (\ref{E4}), we get
\begin{equation}\label{E6}
\lim_{k\to\infty}  d(x_{m(k)-1}, x_{n(k)-1}) =\varepsilon.
\end{equation}
Next, we apply inequality (\ref{ctr}) for $x = x_{m(k)-1}$ and $y = x_{n(k)-1}$, we obtain
\begin{equation}\label{E7}
\psi(a_{k})\le \phi(b_{k}),
\end{equation}
where $a_{k}=d(x_{m(k)},x_{n(k)})$ and $b_{k}=d(x_{m(k)-1},x_{n(k)-1})$.
Hence,  by using (\ref{E4}), (\ref{E6}) and {(ii)} from Definition \ref{shift}, we obtain from (\ref{E7}) that  $\varepsilon=0$, which is a contradiction. This shows that $\{x_n\}$ is a Cauchy sequence. Furthermore, since $X$ is a complete metric space, then  $\{x_n\}$ converges to some $z\in X$.\\
{\it Step 3.} we shall prove that $z$ is a fixed point of $T$.\\
We set $x = x_{n}$ and $y = z$ in (\ref{ctr}), we obtain
\begin{equation*}
\psi(d(x_{n+1},Tz))\le \phi(d(x_n,z)),
\end{equation*}
which implies by {(i)} from Definition \ref{shift} that
\begin{equation*}
d(x_{n+1},Tz)\le d(x_n,z),
\end{equation*}
Letting $n\to\infty$ in the above inequality, we obtain that $d(z,Tz)=0$, that is, $z=Tz$.\\
{\it Step 4.} we shall prove the uniqueness of the fixed point.\\
Suppose that there exist two fixed points $z_1$ and $z_2$ in  $X$, that is, $Tz_1=z_1$ and $Tz_2=z_2$.\\
We set $x = z_1$ and $y = z_2$ in  (\ref{ctr}), we get
$$
\psi(d(z_1, z_2))=\psi(d(Tz_1,Tz_2))\le \phi(d(z_1,z_2)),
$$
which implies, by {(ii)} from Definition (\ref{shift}), that $d(z_1,z_2)=0$, that is, $z_1=z_2$.
\end{proof}

\begin{proposition}
Theorem \ref{DC} is a particular case of Theorem \ref{TH1}.
\end{proposition}

\begin{proof}
Suppose that all conditions of Theorem \ref{DC} are satisfied, then we shall prove that all conditions of  Theorem \ref{TH1} are satisfied too. To this end, suppose there exist two altering distance functions $\psi$ and $\varphi$ such that the contraction (\ref{cnt}) holds. We pose that $\phi(t)=(\psi-\varphi)(t)$, then we have to prove that the pair of functions $(\psi,\phi)$ is a pair of shifting functions. Let $u,v\in[0,+\infty)$, if $\psi(u)\le \phi(v)=(\psi-\varphi)(v)\le\psi(v)$, then $u\le v$ since $\psi$ is a nondecreasing function. This proves that the condition {(i)} of Definition \ref{shift} \ref{TH1} is satisfied. Next,  let  two convergent sequences $\{u_n\}$ and $\{u_n\}$  such that $\lim\limits_{n\to\infty}u_n=\lim\limits_{n\to\infty}v_n=w$, and suppose that $\psi(u_n)\le\phi(v_n)$ for all $n$. Hence, by letting $n\to\infty$ in $\psi(u_n)\le \phi(v_n)=(\psi-\varphi)(v_n)$ and using the continuity of $\psi$ and $\varphi$, we get $\psi(w)\le\psi(w)-\varphi(w)$, which implies that $\varphi(w)=0$ so $w=0$.  This proves that the condition {(ii)}  of Definition \ref{shift} is satisfied.  Finally, the contraction (\ref{ctr}) follows immediately from (\ref{cnt}). Consequently, the  Theorem \ref{DC} is a particular case of Theorem \ref{TH1}.
\end{proof}

\begin{example}
Let $X=[0,1]\cup\{2,3,4,\dots\}$ and define the metric on $X$ by
\begin{align*}
d(x,y)=\left\{
  \begin{array}{ll}
    |x-y|, & \text{ if } x,y\in [0,1] \text{ and } x\ne y; \\
       x+y, & \text{ if } (x,y)\not\in [0,1]\times[0,1] \text{ and } x\ne y; \\
             0 & \text{ if }  x= y.
  \end{array}
\right.
\end{align*}
We know that $(X,d)$ is a complete metric space (see \cite{BW}).\\
Define the functions $\psi$ and $\phi$ as follows
\begin{align*}
\psi(x)=\left\{
  \begin{array}{ll}
    \ln(\dfrac{1}{12}+\dfrac{5}{12}x), & \text{ if } x\in [0,1]; \\[2mm]
    \ln(\dfrac{1}{12}+\dfrac{4}{12}x), & \text{ if } x >1. 
  \end{array}
\right.
\end{align*}
and
\begin{align*}
\phi(x)=\left\{
  \begin{array}{ll}
    \ln(\dfrac{1}{12}+\dfrac{3}{12}x), & \text{ if } x\in [0,1]; \\[2mm]
     \ln(\dfrac{1}{12}+\dfrac{2}{12}x), & \text{ if } x >1.
  \end{array}
\right.
\end{align*}
Let $T:X\to X$ defined by 
\begin{align}\label{T}
Tx=\left\{
  \begin{array}{ll}
   \displaystyle\frac{1}{5} x, & \text{ if } x\in [0,1[; \\[2mm]
   \displaystyle\frac{3}{125} , & \text{ if } x \in\{1,2,3,\dots\}.
  \end{array}
\right.
\end{align}
Without loss of generality, we assume that $x > y$ and discuss the following cases.\\
\noindent\textbullet ~{\it Case 1,  $x\in[0,1]$}:
\begin{align*}
 \psi(d(Tx,Ty)) & = \ln(\frac{1}{12}+ \frac{5}{12} d(Tx,Ty))  \\
                             & = \ln(\frac{1}{12}+ \frac{5}{12}|Tx-Ty|)  \\
                             & = \ln(\frac{1}{12}+\frac{1}{12}|x-y|)  \\
                             & \le  \phi(d(x,y))           
\end{align*}
\noindent\textbullet~{\it Case 2, $x\in \{2,3,\dots\}$}:\\
\indent If $y\in [0, 1[$, then
\begin{align*}
  \psi(d(Tx,Ty)) & = \ln(\frac{1}{12}+ \frac{5}{12}d(Tx,Ty))  \\
                             & = \ln(\frac{1}{12}+ \frac{5}{12}|Tx-Ty|)  \\
                             & \le  \ln(\frac{1}{12}+ \frac{5}{12}(\frac{3}{125}+\frac{1}{5}y))  \\
                             & \le   \ln(\frac{1}{12}+\frac{1}{100}+\frac{1}{12}y) \\
                             & \le  \phi(d(x,y))\quad(\text{ since }\frac{1}{100}+\frac{1}{12}y\le \frac{1}{12}(x+y) )
\end{align*}
\indent  If $y\in \{1,2,3,\dots\}$, then
\begin{align*}
  \psi(d(Tx,Ty)) & \le \ln(\frac{1}{12}+ \frac{5}{12}d(Tx,Ty)) \\
                             & = \ln(\frac{1}{12})  \\
                             & \le  \phi(d(x,y))
\end{align*}

\end{example}

\begin{remark}
We cannot apply Theorem \ref{DC} to prove the existence and uniqueness of fixed point of the mappings $T$ defined in (\ref{T}), however we can apply  Theorem \ref{TH1}.
\end{remark}

\end{document}